\documentclass[reqno, 10pt, a4paper]{amsart}

\usepackage{fullpage}
\selectfont

\usepackage[utf8]{inputenc}
\usepackage[OT2,T1]{fontenc}
\usepackage[english]{babel}

\usepackage{amsmath}
\usepackage{amssymb}
\usepackage{amsthm}

\usepackage[mathcal]{euscript}

\usepackage[dvipsnames]{xcolor}

\usepackage{hyperref}

\usepackage{colonequals}

\theoremstyle{plain}
\newtheorem{theorem}{Theorem}[section]
\newtheorem{proposition}[theorem]{Proposition}
\newtheorem{lemma}[theorem]{Lemma}
\newtheorem{corollary}[theorem]{Corollary}

\theoremstyle{remark}
\newtheorem{remark}[theorem]{Remark}

\numberwithin{equation}{section}

\newcommand\eps{\varepsilon}
\renewcommand\epsilon{\varepsilon}

\DeclareMathOperator{\Aut}{Aut}
\DeclareMathOperator{\diag}{diag}
\DeclareMathOperator{\disc}{disc}
\DeclareMathOperator{\GL}{GL}
\DeclareMathOperator{\USp}{USp}
\DeclareMathOperator{\Jac}{Jac}
\DeclareMathOperator{\Tr}{Tr}

\newcommand\Fbar{\overline{\F}}

\newcommand\F{\mathbb{F}}
\newcommand\Fq{\F_q}
\newcommand\PP{\mathbb{P}}
\newcommand\R{\mathbb{R}}
\newcommand\Z{\mathbb{Z}}

\newcommand\Cc{\mathcal{C}}
\newcommand\Hc{\mathcal{H}}
\newcommand\M{\mathcal{M}}
\newcommand\X{\mathcal{X}}

\newcommand\fraka{\mathfrak{a}}

\newcommand{\Fr}{\textup{Fr}}

\hypersetup{
  pdfauthor   = {Bergstr\"om, Howe, Lorenzo Garc\'ia, Ritzenthaler},
  pdftitle    = {Lower bounds on the maximal number of rational points on curves over finite fields},
  pdfsubject  = {},
  pdfkeywords = {},
  backref=true, pagebackref=true, hyperindex=true, colorlinks=true,
  breaklinks=true, urlcolor=blue, linkcolor=blue, citecolor=blue,
  bookmarks=true, bookmarksopen=true}

\makeatletter
\@namedef{subjclassname@2020}{%
  \textup{2020} Mathematics Subject Classification}
\makeatother

\begin{document}

\title{Lower bounds on the maximal number of rational points on curves over finite fields}
\date{April 29, 2022}

\author[Bergstr\"om]{Jonas Bergstr\"om}
\address{%
Jonas Bergstr\"om, 
Matematiska institutionen, Stockholms Universitet, SE-106 91, Stockholm, Sweden
}
\email{jonasb@math.su.se}

\author[Howe]{Everett W. Howe}
\address{%
Everett W. Howe,
Independent mathematician,
San Diego, CA 92104 U.S.A.
}
\email{however@alumni.caltech.edu}
\urladdr{\href{https://ewhowe.com}{https://ewhowe.com}}

\author[Lorenzo]{Elisa Lorenzo García}
\address{%
	Elisa Lorenzo Garc\'ia,
  Universit\'e de Neuch\^atel, rue Emile-Argand 11, 2000, Neuch\^atel,
  Switzerland. 
}

\address{%
	Elisa Lorenzo Garc\'ia,
   Univ Rennes, CNRS, IRMAR - UMR 6625, F-35000
 Rennes, %
  France. %
}
\email{elisa.lorenzo@unine.ch, elisa.lorenzogarcia@univ-rennes1.fr}

\author[Ritzenthaler]{Christophe Ritzenthaler}
\address{%
	Christophe Ritzenthaler,
  Univ Rennes, CNRS, IRMAR - UMR 6625, F-35000
 Rennes, %
  France. %
  }
  
\address{%
	Christophe Ritzenthaler,
  Université Côte d'Azur, CNRS, LJAD UMR 7351,
  Nice,
  France
}
\email{christophe.ritzenthaler@univ-rennes1.fr}






\begin{abstract}
For a given genus $g \geq 1$, we give lower bounds for the maximal number of rational points on a smooth projective absolutely irreducible curve of genus $g$ over $\F_q$. As a consequence of Katz--Sarnak theory, we first get for any given $g>0$, any $\epsilon>0$ and all $q$ large enough, the existence of a curve of genus $g$  over $\F_q$ with at least $1+q+ (2g-\epsilon) \sqrt{q}$ rational points. Then using sums of powers of traces of Frobenius of hyperelliptic curves, we get a lower bound of the form $1+q+1.71 \sqrt{q}$ valid for $g \geq 3$ and odd $q \geq 11$. Finally, explicit constructions of towers of curves improve this result, with a bound of the form $1+q+4 \sqrt{q} -32$ valid for all 
$g\ge 2$ and for all~$q$.
 \end{abstract}

\maketitle


\section{Introduction}

Researchers who study $N_q(g)$, the maximal number of rational points on curves\footnote{ 
   Throughout this paper, the word `curve' will always mean a
   projective, absolutely irreducible, smooth variety of dimension~$1$.}
of genus $g$ over a finite field $\F_q$, generally follow the lead of Serre's 1985 Harvard lectures \cite{serre-book} and focus on two cases: one in which $q$ is fixed and the genus goes to infinity, and one in which $g$ is fixed and $q$ varies. 
In the first case, a great number of results have been achieved that control the asymptotic behavior of the ratio $N_q(g)/g$; see \cite{beelen2022survey} and the references there. In the second case, while a closed formula is known for $N_q(0), N_q(1)$ \cite{deuring} and $N_q(2)$ \cite{serre-point}, the exact nature of the growth of $N_q(g)$ for fixed $q$ remains open in general.  One of the tantalizing challenges already proposed in \cite[Sec.~4.3]{serre-book} is to determine whether for every $g$, the value $N_q(g)$ remains at a bounded distance from the Hasse--Weil bound $1+q+2g\sqrt{q}$ for all~$q$, as is the case for $g=0, 1$ and $2$. It is hard even to get good heuristics for this question (see our attempt in remark~\ref{rem:heuristic}), and in a recent personal communication J-P. Serre raised a less ambitious question: Is it possible to give for each $g$ a positive constant $c$ such that for all sufficiently large $q$, we have $N_q(g)\ge 1+q +c \sqrt{q}$\,? In this paper we provide several methods that lead to a positive answer to the question, even when we limit our consideration to hyperelliptic curves.\\

Serre not only asked this question but also suggested that it might be answered through the consideration of the (weighted) sum $S_n(q,\Hc_g)$ of the $n$-th powers of the traces of Frobenius of genus-$g$ hyperelliptic curves over~$\F_q$, for even~$n$. This strategy is at the root of the computations we present in Section~\ref{sec:moments}, which is chronologically the first path we followed. Using the explicit formula from \cite{berg2} for $n=6$, one finds for instance that for $g \geq 3$ and $q \geq 25$, we have $N_q(g) \geq 1+q+1.55 \sqrt{q}$.\\

We then quickly realized that by using Katz--Sarnak theory, one can actually get an optimal result of this flavor, namely that for every  $\epsilon>0$ and every $g$, one can find a $q_0$ such that for $q>q_0$, one has $N_q(g) \geq 1+q+(2g-\epsilon) \sqrt{q}$ (see Corollary~\ref{cor:2g-e}). We were surprised to find no trace of this result in the existing literature, as it can be derived easily. Notice though that a related result appears already in \cite[rem.~1.1.2]{serre-book}: For every  curve $C/\F_q$ and $\epsilon>0$, we have $\lvert\#C(\F_{q^i}) - (q^i+1)\rvert \ge (2g-\epsilon) \sqrt{q^i}$ for infinitely many $i$.

One drawback of the method used in Section~\ref{sec:KS} in comparison with the one in Section~\ref{sec:moments} is of course that the value of $q_0$ is unknown. It is then tempting to push the method of Section~\ref{sec:moments} using $S_n(q,\Hc_g)$ further on, in particular since one shows in Theorem~\ref{theorem:limitH0} that the limit when $n$ goes to infinity of
\[
\fraka_{2n}(g)
\colonequals \left(\lim_{q \to \infty} 
             \frac{S_{2n}(q,\Hc_g)}{q^{2g-1+i/2}}\right)^{\frac{1}{2n}}
\]             
is $2g$. This  proves that the strategy using moments provides the same asymptotic bound as Katz--Sarnak. But the lack of an explicit formula for $S_{2n}(q,\Hc_g)$ when $n$ is large also prevents us from giving an explicit value for~$q_0$. Interestingly though, one can show that $\fraka_{2n}(g)$ can also be efficiently computed using a representation of $\USp_{2g}$;  see Theorem~\ref{theorem:H0}.\\

In Section~\ref{sec:explicit} we develop yet another approach, which surpasses our best lower bounds from Section~\ref{sec:moments}: We show that under relatively mild hypotheses, a hyperelliptic curve $C$ of genus~$g$ can be covered by hyperelliptic curves of genus $2g$ and $2g+1$ that have at least as many rational points as does~$C$. By starting with well-chosen hyperelliptic curves of genus~$2$ and~$3$ with at least $1+q + 4\sqrt{q} - 32$ points, we can recursively construct a hyperellipic curve of any desired genus that has at least this many rational points. This lower bound is obviously smaller than the one we obtain using the Katz--Sarnak approach, but it is better than the one we find using the first few values of $S_n(q,\Hc_g)$, and it applies to all~$q$. The method also suggests an algorithm for producing an explicit curve of any given genus that surpasses this lower bound; see Section~\ref{SS:construction}.\\

\subsection*{Acknowledgement}
The idea of using sums of powers of traces of Frobenius over the collection of all genus-$g$ hyperelliptic curves, which we use in Section~\ref{sec:moments}, was suggested to us in a personal communication by J-P.~Serre. We want to thank him for his generous input, which initiated this whole project. We also thank Markus Kirschmer for his help with the mass formulae in remark~\ref{rem:heuristic}.
Finally, we thank 
Jeff Achter,
Anna Cadoret,
Adrian Diaconu,
and
Rachel Pries
for their helpful comments on the first version of this paper.

\section{Lower bounds from Katz--Sarnak distribution}
\label{sec:KS}
For $g \geq 2$, let $\M_g$ denote the (coarse) moduli space of smooth projective genus-$g$ curves, and let $\Hc_g$  denote the subspace of  hyperelliptic curves. Note that  $\M_g(\F_q)$  and $\Hc_g(\F_q)$ consist of $\Fbar_q$-isomorphism classes of curves that have models over~$\F_q$. Let $\M'_g(\F_q)$ and $\Hc'_g(\F_q)$ denote the sets of $\F_q$-isomorphism classes of (the appropriate types of) curves over~$\F_q$. 

\begin{theorem}[Katz--Sarnak, {\cite[Th.~10.7.12 and~10.8.2]{katz-sarnak}}] \label{th:2gbound}  Fix $g\geq 2$. 
Let $m$ be the Haar measure on $\mathrm{USp}_{2g}$, the compact symplectic group. If $f$  is a continuous function on $\mathrm{USp}_{2g}$ that is constant on the conjugacy classes, then \[
\int_{m \in \mathrm{USp}_{2g}} f(m) \,dm 
= \int_{\theta \in  [0,\pi]^g} f \circ h(\theta) \,d\mu_g(\theta) 
= \frac{1}{\# \M_g(\F_q)} \cdot 
  \sum_{C \in \M'_g(\F_q)} \frac{f\circ h(\theta_C)}{\# \Aut_{\F_q}(C)}
  +O(q^{-1/2})\,,
\]
where $\theta_C=(\theta_1,\ldots,\theta_g) \in [0,\pi]^g$ are the Frobenius angles of the Jacobian of $C$, where $h$ is the function given by
\[
h(\theta_1,\ldots,\theta_g) 
= \diag(e^{ i\theta_1},\ldots,e^{ i \theta_g},
        e^{-i\theta_1},\ldots,e^{-i \theta_g})\,,
\]
and where 
$d\mu_g(\theta) = \delta_g(\theta) \,d\theta_1\cdots d\theta_g$ is the density measure with 
\[
\delta_g(\theta)
= \frac{1}{g!} \,
  \prod_{j<k}\bigl(2\cos(\theta_j)-2\cos(\theta_k)\bigr)^2 \,
  \prod_j\Bigl(\frac{2}{\pi}\sin^2(\theta_j)  \Bigr)\,.
\]
The same holds when one replaces $\M_g$ with $\Hc_g$. \end{theorem}

As a consequence of Theorem~\ref{th:2gbound}, we obtain the following result. For $\mathcal{M}_g$ this is~\cite[Cor.~4.3]{LachaudDis}; the proof for $\Hc_g$ follows the same argument presented in~\cite{LachaudDis}.

\begin{proposition}\label{prop:dist}
Fix $g\geq 2$. 
For a genus-$g$ curve we write $\#C(\F_q)=1+q+\tau(C)\sqrt{q}$. Then we have
\[
\frac{\# \{C\in \mathcal{M}'_g(\F_q),\,\tau(C)\leq x\}}{\# \mathcal{M}_g(\F_q)}=F(x)+O(q^{-1/2})\,,
\] 
where $F(x)=\int_{A_x} d\mu_g$  and $A_x=\{(\theta_1,...,\theta_g)\in[0,\pi]^g:\,\sum_{j}2\operatorname{cos}\theta_j\leq x\}$. The same holds when one replaces $\mathcal{M}_g$ with $\Hc_g$. \qed
\end{proposition}

\begin{corollary}\label{cor:2g-e} Fix $g\geq 2$ and $\epsilon>0$. For all sufficiently large~$q$, there exist hyperelliptic genus-$g$ curves $C/\F_q$ and $C'/\F_q$ with $\#C(\F_q)\geq1+q+(2g-\epsilon)\sqrt{q}$ and $\#C'(\F_q)\leq 1+q-(2g-\epsilon)\sqrt{q}$. 
\end{corollary}

\begin{proof}
Applying Proposition~\ref{prop:dist} to $\Hc_g$, 
we find that for $q$ large enough there exists $c>0$ such that
\[
\frac{\# \{C\in \Hc'_g(\F_q),\,\tau(C)\leq 2g-\epsilon\}}{\# \Hc_g(\F_q)}-F(2g-\epsilon) \leq \frac{c}{q^{1/2}}\,.
\]
Since $F\colon\R\rightarrow[0,1]$  does not depend on $q$ and is a continuous, non-negative, strictly increasing function on $[-2g,2g]$ with $F(2g)=1$, we have also, for $q$ large enough,
\[
F(2g-\epsilon)\leq 1-\frac{c}{q^{1/2}}-\frac{1}{q^{2g-1}}\,.
\]
Hence,
\[
\frac{\# \{C\in \Hc'_g(\F_q),\,|\tau(C)|\leq 2g-\epsilon\}}{\# \Hc_g(\F_q)} \leq F(2g-\epsilon)+ \frac{c}{q^{1/2}} \leq 1-\frac{1}{q^{2g-1}}\,.
\]
Since $\#\Hc_g(\F_q)=q^{2g-1}$ \cite[Prop.~7.1, p.~87]{BG01}, multiplying the previous inequality by this cardinality, we see that there exists a hyperelliptic curve $C/\F_q$ with $\#C(\F_q)\geq1+q+(2g-\epsilon)\sqrt{q}$ points. Taking its quadratic twist $C'$ gives the other inequality.
\end{proof}

\begin{remark}
Actually, the theory of Katz--Sarnak (in particular \cite[9.6.10]{katz-sarnak}) allows us to prove a similar result for low-dimensional families of curves. For a given prime $\ell$, let $S$ be a connected normal scheme, separated and of finite type over $\Z[1/\ell]$. Let $X/S$ be a smooth scheme such that for each finite field $k$ of characteristic different from $\ell$ and each point $s \in S(k)$, $X_s/k$ is a curve of genus $g$. We assume that for each $X_s$ the geometric monodromy group satisfies \cite[9.3.7.2]{katz-sarnak}; in particular, it is conjugate in $\GL_{2g}$ to a constant group. Examples of such families include for instance the generic element of the one-parameter families $y^2=f(x) (x-t)$ of hyperelliptic curves \cite{hall} or cyclic triple covers of~$\PP^1$ \cite{pries-achter}.

Now we simply assume that the dimension of the image of $S \otimes \bar{k}$ in the moduli space is at least~$1$. Then there exists 
a function $g(q)$ going to infinity such that $\#S(\F_q) \geq g(q)$. Moreover, Katz--Sarnak implies that for $x \in [-2g,2g]$,  
\[
\frac{\# \{s\in S(\F_q),\,\lvert\tau(X_s)\rvert\leq x\}}{\# S(\F_q)} - F(x)\leq f(q)
\]
for a strictly increasing distribution function $F(x)$ and a decreasing function $f(q)$ going to $0$.

Given $\epsilon>0$, let $q$ be large enough so that $F(2g-\epsilon)\leq 1-f(q)-\frac{1}{g(q)}$. Then  
\[
\frac{\# \{s \in S(\F_q),\,\lvert\tau(X_s)\rvert\leq 2g-\epsilon\}}{\# S(\F_q)}
  \leq 1-\frac{1}{g(q)}
  \leq \frac{\# S(\F_q)-1}{\# S(\F_q)}\,,
\]
which shows there is at least one curve in the family $X/S$ with more than $1+q+(2g-\epsilon)\sqrt{q}$ points. More generally, this argument can be adapted to prove the existence of a curve with number of points in the interval $[1+q+(a-\epsilon)\sqrt{q},1+q+(a+\epsilon)\sqrt{q}]$ for any $a$ with $-2g \leq a \leq 2g$.
\end{remark}

\begin{remark} \label{rem:heuristic}
 The lower bound of Corollary~\ref{cor:2g-e}  is of course far from answering Serre's question on the existence of a curve with bounded defect.  Even the question of the existence of infinitely many $p$ for which there exist defect-$0$ curves of a fixed genus $g$ over $\F_p$ is open. The following is a naive attempt to make up our mind on a direction to take for this challenge. The Jacobians of such curves are isogenous to powers of an ordinary elliptic curve $E/\F_p$ with trace $-\lfloor 2 \sqrt{p} \rfloor $. It is tempting to look at the number of principally polarized abelian varieties of this type up to isomorphism. Using the equivalence of categories \cite[Cor.~3.6]{KNRR}, this is the same as counting, up to isometry, unimodular positive definite hermitian $R$-lattices of rank $g$ where $R=\Z[x]/(x^2+\lfloor 2 \sqrt{p} \rfloor x + p)$. When $R$ is maximal, crude estimations of the mass formulae and class numbers of \cite{siegel, siegel3}  kindly provided by \cite{kirschmer-perso}
show that their number should be smaller than $C(g) \times \disc(R)^{g(g+1)/4 + o(1)}$, where $C(g)$ is a constant that does not depend on $R$. Since $\disc(R)$ is smaller than $4 \sqrt{p}$, we get at most $p^{g(g+1)/8 + o(1)}$ distinct principally polarized abelian varieties isogenous to the power of an elliptic curve with trace $-\lfloor 2 \sqrt{p} \rfloor $. If one makes the assumption that Jacobians over $\F_p$ are well distributed among principally polarized abelian varieties over $\F_p$, the `chance' to fall in the Jacobian locus  may be estimated as $p^{(3g-3)-g(g+1)/2}$. Hence  for $g>6$, it is heuristically unlikely to find a defect-$0$ curve over $\F_p$ when $p$ is large. We believe that the assumption that $R$ be maximal is not necessary, but proving this would require a better understanding of the mass formulae for non-projective $R$-lattices.
\end{remark}


\section{Lower bound from explicit power of traces} \label{sec:moments}
\subsection{Weighted trace powers}
Let us define the correct moments we want to compute. If $C/\F_q$ is a genus-$g$ curve, we denote by $[C]$ the set of representatives of its twists and define 
\begin{equation} \label{eq:ni}
s_n(C) = \sum_{C' \in [C]} \frac{(q+1-\#C'(\F_q))^n}{\# \Aut_{\F_q}(C')}\,.
\end{equation}

The following result is probably well known, but we provide a proof for lack of proper reference (note that the case  $s_0(C)=1$ can be found in \cite[Prop.5.1]{VdG92}).
\begin{proposition} \label{prop:integer}
For every curve $C/\F_q$ and every $n \geq 0$, $s_n(C)$ is an integer.
\end{proposition}

For this, we will need some elementary lemmas.  As in \cite[Prop. 9]{MR2678623}, let us define for $g\in \Aut_{\Fbar_q}(C)$, the set $[g]_{\Fr}$ of elements $h$ such that there exists $x\in \Aut_{\Fbar_q}(C)$ with $h=xg\,^{\Fr}\!x^{-1}$, where $^{\Fr}$ is the geometric $\F_q$-Frobenius morphism (acting here on $x^{-1}$). To a given set $[g]_{\Fr}$, one can associate a twist $C'$ of $C$.

\begin{lemma}[{\cite[Proof of Prop.~5.1]{VdG92}}] \label{lemma:prod}
Let $g\in \Aut_{\Fbar_q}(C)$ and let $C'$ be the twist associated to the Frobenius conjugacy class $[g]_{\Fr}$ of $g$. Then $\#\Aut_{{\F_q}}(C')\cdot \#[g]_{\Fr}=\#\Aut_{\Fbar_q}(C)$. \qed
\end{lemma}

\begin{lemma}  \label{lemma:G}
Let $K$ be a field of characteristic $0$, let $n$ and $k$ be positive integers and let $G$ be a finite subgroups of $\GL_k(K)$. 
\begin{enumerate}
\item For every $A,B\in M_n(K)$ we have $(A\cdot B)^{\otimes k}=A^{\otimes k}\cdot B^{\otimes k}$. 
\item Hence, the map $g \mapsto g^{\otimes n}$ from $G$ to  $G^{\otimes n} \subseteq \GL_{kn}(K)$ induces a surjective morphism on its image $G_n$.  
\item The matrix $P_{G_n}=\frac{1}{\#G}\sum_{g\in G}g^{\otimes n}= \frac{1}{\#G_n}\sum_{g\in G_n} g$ is a projection. Hence its eigenvalues are $0$ and $1$. 
\qed
\end{enumerate}
\end{lemma}

By abuse of notation we denote  by $g\in\Aut_{\Fbar_q}(C)$ the corresponding element in $\Aut_{\Fbar_q}(T_{\ell} \Jac C)$ that we see as a matrix of size $2g\times 2g$ with coefficients in $\Z_{\ell}$ for a prime $\ell\neq p$. We denote $\pi_C$ the Frobenius endomorphism of $\Jac C$ or its matrix for the action on $T_{\ell} \Jac C$ in some arbitrary basis. Let us recall from \cite[Prop. 11]{MR2678623} that  if $C'$ is a twist of $C$ given by an element $g\in \Aut_{\Fbar_q}(C)$  then $\pi_{C'}=\pi_{C} \cdot g$. 

\begin{proof}[Proof of Proposition~\textup{\ref{prop:integer}}] One has
\begin{align*}
s_n(C) &=  \sum_{C' \in [C]} \frac{(q+1-\#C'(\F_q))^n}{\# \Aut_{\F_q}(C')}
=\sum_{g \in \Aut_{\Fbar_q}(C)} \frac{(q+1-\#C'(\F_q))^n}{\# \Aut_{\Fbar_q}(C)}
=\sum_{g \in \Aut_{\Fbar_q}(C)} \frac{\Tr(\pi_{C'})^n}{\# \Aut_{\Fbar_q}(C)} \\
&= \sum_{g \in \Aut_{\Fbar_q}(C)} \frac{\Tr((\pi_{C'})^{\otimes n})}{\# \Aut_{\Fbar_q}(C)}=\sum_{g \in \Aut_{\Fbar_q}(C)} \frac{\Tr((\pi_C \cdot g)^{\otimes n})}{\# \Aut_{\Fbar_q}(C)}
=\sum_{g \in \Aut_{\Fbar_q}(C)} \frac{\Tr(\pi_C^{\otimes n}\cdot g^{\otimes n})}{\# \Aut_{\Fbar_q}(C)} \\
&= \Tr\Biggl(\pi_C^{\otimes n} \cdot \frac{1}{\# \Aut_{\Fbar_q}(C)} \sum_{g \in \Aut_{\Fbar_q}(C)}  g^{\otimes n}\Biggr) 
=\Tr(\pi_C^{\otimes n} \cdot P_{G_n})\,.
\end{align*}

The first equality is by definition, the second one by Lemma \ref{lemma:prod}, the third one by the Weil Conjectures, the fourth one is a classical property of the tensor product of matrices (see for instance \cite[Chap.~2 Prop.~2]{serre-linear}), the fifth one is \cite[Prop. 11]{MR2678623}, the sixth one is Lemma \ref{lemma:G}(i), the seventh is the commutativity of the sum and the trace of matrices and the eighth is Lemma \ref{lemma:G}(ii). 

Now, since for every $g \in \Aut_{\Fbar_q}(C)$ we have $g \circ \Fr = \Fr \circ g'$ for some $g'  \in \Aut_{\Fbar_q}(C)$, we see that $\pi_C^n$ commutes with $P_{G^{\otimes n}}$,  hence the eigenvalues of their product are the product of the eigenvalues. It is well-known that the eigenvalues of $\pi_C$, and therefore of $\pi_C^{\otimes n}$, are algebraic integers, and the ones of $P_{G_n}$ are $0$ or $1$. We conclude that the trace is an algebraic integer over $\Z_{\ell}$ as well.  Now the conclusion follows since we know that $s_n(C)$ is also a rational number, so it must be an integer.
\end{proof}

For $\X=\M_g$ or $\Hc_g$
we denote by $S_{n}(q,\X)$ the sum of the $s_n(C)$ when $C$ runs over the set of $\Fbar_q$-isomorphism classes of curves $C$ in $\X$ over $\F_q$.
Then, for two polynomials $f$ and $g$, let $[f/g]$ denote the polynomial quotient in the Euclidian division of $f$ by $g$. 
In Theorem~\ref{theorem:Si} and Remark~\ref{remark:S8}, the polynomial quotients correspond to the stable part of the cohomology, see \cite[Section 1.5]{petersenetal} together with \cite[Section 13]{berg2}.

\begin{theorem} \label{theorem:Si}  For every $g \geq 2$ and prime power $q$ we have 
\begin{align*}
S_2(q,\Hc_g)&= \, [q^{2g}]-1 \\
S_4(q,\Hc_g)&= \left [\frac{q^{2g} (3q^2+q+1)}{q+1} \right]-\frac{1}{2}(q-1)(q-2)(q+1)g^2+\frac{1}{2}(-q^3+2q^2-7q+2)g-3q+2 \\
S_6(q,\Hc_g)&=  \left [\frac{q^{2g}(15q^4+16q^2+2q+1)}{(q+1)^2}\right]-\frac{1}{24}(q-1)(q-3)(q+1)(q^3-6q^2+4q+13)g^4 \\ 
&\qquad -\frac{1}{12}(q+1)(q-3)^2(q^3-4q^2+18q-3)g^3\\
&\qquad+\frac{1}{24}(q^6-9q^5-99q^4+382q^3-469q^2+491q-9)g^2 \\
&\qquad+\frac{1}{12}(q^6-9q^5-19q^4+78q^3-423q^2+567q-723)g-15q^2+30q-61\\
&\qquad-\delta_{q} \, \frac{5}{8}g(g-1)(g-2)\bigl((g-3)(q-1)-4\bigr)\,,
\end{align*}
where $\delta_q$ is equal to $1$ if $q \equiv 0 \bmod 2$ and $0$ otherwise.
\end{theorem}
\begin{proof} In \cite[Section 7,10]{berg2} it is described how to compute $S_i(q,\Hc_g)$ for $i=2,4$ and $6$. Note that in the notation of \cite{berg2}, $S_i(q)$ equals $a_{1^i}|_g$. \end{proof}

\begin{remark} \label{remark:S8}
From \cite[Section 7,10]{berg2} we see that the information missing to compute $S_8(q,\Hc_g)$ for any $g \geq 3$, is $S_8(q,\M_{1,1})$, $S_8(q,\Hc_2)$ and 
\begin{equation} \label{eq:ram}
 \sum_{(C,p) \in \M'_{1,1}(\F_q)} \frac{(q+1-\#C'(\F_q))^6 (q+1-r_1(C))^2}{\# \Aut_{\F_q}(C)}\,, 
\end{equation}
in the notation of \cite[Section 12]{berg2}. Here, $q+1-r_1(C)$ is the number of ramification points over  $\F_q$ of $(C,p)$ as a double cover of $\mathbb P^1$. For every $q$, $S_8(q,\Hc_2)$ can be determined through the information in \cite[Theorem 2.1]{petersen}. A level two structure on an elliptic curve can be described in terms of a marking of its ramification points. For odd $q$, the cohomology of local systems on $\M_{1,1}[2] \otimes \Fbar_q$ (the moduli space of elliptic curves with a level two structure) and its structure as a representation of $\mathrm{SL}(2,\F_2) \cong \mathbb S_3$, is well known via the Eichler--Shimura isomorphism, see~\cite{deligne3} and \cite[Theorem 6]{faltings}. Using this we can also compute \eqref{eq:ram}. Putting these results together we find that for every $g \geq 2$ and odd $q$, we have 
\begin{flalign*}
S_8(q,\Hc_g) 
&=\rlap{$\displaystyle\left[\frac{q^{2g}(105q^6-105q^5+273q^4-83q^3+66q^2+3q+1)}{(q+1)^3}\right]$} \\
&\qquad -\frac{1}{720}
\rlap{$(q-1)(q-3)(q-4)(q-5)(q+1)^2(q^3-9q^2+15q+33) g^6$} \\ 
&\qquad -\frac{1}{240}
\rlap{$(q-3)(q-5)(q+1)(q^6-13q^5+92q^4-280q^3+215q^2+565q-100)g^5$} \\
&\qquad +\frac{1}{144}
\rlap{$(q-3)(q+1)(q^7-18q^6-11q^5+828q^4-3455q^3+5826q^2-4947q+48)g^4$} \\
&\qquad +\frac{1}{48}
\rlap{$(q-3)(q^8-17q^7+55q^6+61q^5-1329q^4+3573q^3-3219q^2+2095q+124)g^3$} \\
&\qquad +\frac{1}{360}
\rlap{$(-2q^9+40q^8+19q^7-2480q^6-3470q^5+52390q^4-166449q^3$} \\
&&+338580q^2-424098 q+453870)g^2& \\
&\qquad +\frac{1}{60}
\rlap{$(-q^9+20q^8-85q^7-120q^6-214q^5+2530q^4-22515q^3+62220q^2$}\\
&& -127725q+201870)g& \\
&\qquad -105q^3+420q^2-1218q+2582\,.
\end{flalign*}
\end{remark}

\begin{theorem} \label{theorem:CSi} For every $g\geq 2$, prime power $q$ and even $n \geq 2$, let $a_{q}=(S_n(q,\Hc_g)/q^{2g-1+n/2})^{1/n} $. There exists a hyperelliptic genus $g$ curve $C/\mathbb{F}_q$ with $\#C(\mathbb{F}_q)\geq1+q+ a_q \sqrt{q}$ and a hyperelliptic curve $C'/\mathbb{F}_q$ with $\#C'(\mathbb{F}_q)\leq 1+q-a_q \sqrt{q}$. 
\end{theorem}
\begin{proof} By the above there are curves $C_1,\ldots,C_{q^{2g-1}}$ over $\F_q$ such that 
$S_n(q)=\sum_{i=1}^{q^{2g-1}} s_n(C_i)$. Since $n$ is even, all $s_n(C_i)$ are non-negative and so there must be a $j$ such that $s_{n}(C_j) \geq S_n(q)/q^{2g-1}$. Since $s_0(C_j)=\sum_{C\in [C_j]} 1/\# \Aut_{\F_q}(C)=1$, $s_n(C_j)$ can be seen as a weighted average, and it then follows that there is a $C \in [C_j]$ such that $(\#C(\F_q)-q-1)^n \geq S_n(q)/q^{2g-1}$. This shows that $\#C(\F_q) \leq q+1-a_q\sqrt{q}$ with $a_q=(S_n(q)/q^{2g-1+n/2})^{1/n} $. The quadratic twist of $C$ gives the curve with the opposite bound. 
\end{proof}

Using the formulas of Theorem~\ref{theorem:Si} and Remark~\ref{remark:S8} we get, for $i=2,4,6$ and $8$,  concrete lower bounds for $(S_n(q,\Hc_g)/q^{2g-1+n/2})^{1/n}$ valid for $q$ large enough.
\begin{corollary} \label{cor:CSi}  
There exists a hyperelliptic curve of genus $g$ over $\F_q$ with $\#C(\F_q) \geq 1+q+a \sqrt{q}$ with
\begin{itemize}
\item $a=(S_4(q)/q^{2g+1})^{1/4} \geq 1.3$  when $g \geq 3$ and $q \geq 13$\textup{;}
\item   $a=(S_6(q)/q^{2g+2})^{1/6} \geq 1.55$ when  $g \geq 3$ and $q \geq 25$\textup{;}
\item  $a=(S_8(q)/q^{2g+3})^{1/8} \geq 1.71$ when $g \geq 3$ and odd $q \geq 11$.
\end{itemize}  
\end{corollary}

One sees that the coefficient of the leading term of $S_n(q)$ controls the growth of the bound on the number of points. This coefficient can be obtained quickly, even when a complete formula for $S_n(q)$ is out of reach, thanks to a relation with representation theory of the compact symplectic group $\mathrm{USp}_{2g}$. 

\begin{theorem} \label{theorem:H0}
For every $g \geq 2$ and even $n \geq 2$ let 
\[
\fraka_n(g)\colonequals
\lim_{q \to \infty} \frac{S_n(q,\X)}{q^{\dim \X+n/2}}
\]
with $\X=\M_g$ or $\Hc_g$. 
Then $\fraka_n(g)$ is equal to the number of times the trivial representation appears in the $\mathrm{USp}_{2g}$-representation $V^{\otimes i}$ with $V$ the standard representation.
\end{theorem}
\begin{proof} 
Using Theorem~\ref{th:2gbound} with $f=\Tr^n$, we see that 
\begin{equation} \label{eq:an}
\fraka_n(g)= \int_{{(\theta_1,\ldots,\theta_g)}\in [0,\pi]^g } \biggl(\sum_{j=1}^g 2 \cos(\theta_j) \biggr)^n d\mu_g(\theta_1,\ldots,\theta_g)\,, 
\end{equation} 
and that this integral also can be written as 
\[
\int_{m \in \mathrm{USp}_{2g}} \mathrm{Tr}(m^{\otimes n})\,dm\,.
\]
By the orthogonality of characters it follows that this integral counts the number of times that the trivial representation appears in the $n$th tensor product of the standard representation. 
\end{proof}
The sequence $\fraka_n(g)$ is called a moment sequence in \cite{kedlayasutherland}; in their notation it equals $M[s_1](n)$. In this article there is an effective formula to compute $\fraka_n(g)$ in terms of a sum of determinants of binomial expressions. However, to prove the following limit result, it is easier to use the integral.

\begin{theorem} \label{theorem:limitH0} For every $g \geq 2$,  we have
\[
\lim_{n \to \infty} (\fraka_{2n}(g))^{1/2n}=2g\,.
\]
\end{theorem}
\begin{proof} First notice that since $\lvert\cos(\theta)\rvert \leq 1$ then $\lvert a_n\rvert \leq (2g)^n$, so $\limsup_{n \to \infty} (\fraka_{2n}(g))^{1/2n}\leq2g$. 

Let us do now the change of variables $t_i=2\cos(\theta_i)$. Then Equation \eqref{eq:an} becomes
\begin{equation} \label{eq:an2}
\fraka_n(g)= \int_{{(t_1,\ldots,t_g)}\in [-2,2]^g } \biggl(\sum_{j=1}^g t_j \biggr)^n\prod_{j=1}^g\sqrt{4-t_j^2}\prod_{1\leq j\leq k\leq g}(t_j-t_k)^2 \,dt_1\ldots dt_g.
\end{equation} 
Since all factors inside the integral of $\fraka_{2n}(g)$ are positive, the value of the integral is greater than the one taken on any sub-domain of $[-2,2]^g$. Fix $0<\epsilon<1$ and define 
\[
I_j=\Bigl[2-2\epsilon+(2j-2)\frac{\epsilon}{2g-1},2-2\epsilon+(2j-1)\frac{\epsilon}{2g-1}\Bigr]\subseteq[2-2\epsilon,2-\epsilon]
\]
for $j=1,\ldots,g$. The sub-domain $S=I_1 \times \cdots \times I_g$ is constructed such that, the values of the $t_i$ are separated by at least $\frac{\epsilon}{2g-1}$ and close to $2$. Then
\begin{align*}
\fraka_{2n}(g) 
&\geq \int_{(t_1,\ldots,t_g) \in S}\biggl(\sum_{j=1}^g t_j \biggr)^{2n}\prod_{j=1}^g\sqrt{4-t_j^2}\prod_{1\leq j\leq k\leq g}(t_j-t_k)^2 \,dt_1\ldots dt_g \\
&\geq   \int_{(t_1,\ldots,t_g) \in S} (2g-2\epsilon g)^{2n} \cdot \epsilon^g \cdot \left(\frac{\epsilon}{2g-1}\right)^{g(g-1)} \,dt_1\ldots dt_g  \\
&\geq \left(\frac{\epsilon}{2g-1}\right)^{g} \cdot (2g-2\epsilon g)^{2n} \cdot \epsilon^g \cdot \left(\frac{\epsilon}{2g-1}\right)^{g(g-1)}.
\end{align*}
So $\liminf_{n \to \infty} (\fraka_{2n}(g))^{1/2n}\geq2g$ and the result follows.
\end{proof}


\section{Lower bounds from explicit constructions} 
\label{sec:explicit}

Corollary~\ref{cor:2g-e} shows that for a fixed genus $g$ and fixed $\eps>0$,
for every large enough $q$ there is a hyperelliptic curve $C/\Fq$ of genus~$g$ 
whose number of points is within $\eps\sqrt{q}$ of the Weil bound. In this 
section we prove a result that is much weaker than this, but that has the 
advantages of working for every $q$ and $g\ge 2$ and of being constructive --- 
see Section~\ref{SS:construction}.

\begin{theorem}
\label{T:explicit}
Let $g>1$ be an integer and let $q$ be a prime power.
\begin{enumerate}
\item \label{T:E1}
      If $q$ is odd, there is a hyperelliptic curve $C/\Fq$ 
      of genus~$g$ with 
      \[
      \#C(\Fq) >  \begin{cases}
                  1 + q + 4\sqrt{q} -  5 & \text{if $q<512$};\\
                  1 + q + 4\sqrt{q} - 32 & \text{if $q>512$}.\\
                  \end{cases}
      \]
\item \label{T:E2}
      If $q$ is even, there is a hyperelliptic curve $C/\Fq$ 
      of genus~$g$ with 
      \[
      \#C(\Fq) >  \begin{cases}
                  1 + q + 4\sqrt{q} -  5 & \text{if $q\le 8$};\\
                  1 + q + 4\sqrt{q} - 12 & \text{if $q>8$}.\\
                  \end{cases}
      \]
\end{enumerate}      
\end{theorem}

\begin{remark}
When $q$ is small with respect to~$g$, there are in fact hyperelliptic curves of
genus $g$ over $\Fq$ having $2q+2$ rational points, the largest number possible
for a hyperelliptic curve over~$\Fq$ of any genus. The sharpest result in this 
direction that we are aware of is~\cite[Theorem~1.6]{Pogildiakov2017}, which
implies that for $g\ge 2$, if $q$ is odd and $q\le 2g+3$, or if $q$ is even and
$q\le g+1$, then there is a hyperelliptic curve of genus~$g$ over $\Fq$ with
$2q+2$ rational points. (The cited result speaks of curves with \emph{no}
points, but the quadratic twist of such a curve has $2q+2$ points.)
\end{remark}

The basic idea of our proof of Theorem~\ref{T:explicit} is to create a tower of
double covers of hyperelliptic curves, each with at least as many rational
points as the one below it. Here is the structure of the argument in the case 
where $q$ is odd: Let $C$ be a hyperelliptic curve of genus $g$ over $\Fq$. If
$C$ has exactly two rational Weierstrass points then we can construct 
hyperelliptic double covers $D$ of $C$, one of genus $2g$ and one of genus
$2g+1$, such that $D$ has exactly two rational Weierstrass points and such that
$\#D(\Fq)\ge \#C(\Fq)$. By using a double-and-add process starting from a curve
of genus~$2$, we can reach every genus whose binary expansion starts with~$10$;
starting from a curve of genus~$3$, we can reach every genus whose binary 
expansion starts with~$11$. Thus, the lower bound we get in the statement of the
theorem is essentially the largest number of points we can obtain on a curve of
genus~$2$ that is suitable as a starting curve for our construction.

In Section~\ref{SS:odd} we flesh out the tower-building argument for odd $q$
sketched in the preceding paragraph. We also explain how to construct 
appropriate base curves of genus~$2$ by gluing together elliptic curves with
many points, and appropriate base curves of genus~$3$ by taking unramified
double covers of such genus-$2$ curves. In Section~\ref{SS:2} we show how to
modify the argument for odd $q$ in order to deal with the fact that in 
characteristic~$2$, hyperelliptic curve are Artin--Schreier extensions of $\PP^1$
rather than Kummer extensions.

The double-cover argument that we use to prove Theorem~\ref{T:explicit} was
inspired by a similar double-cover argument from~\cite{ElkiesHoweEtAl2004},
which shows that if $C$ is a not-necessarily-hyperelliptic curve of genus~$g$
over~$\Fq$, then for every $h\ge 4g$ there is a curve $D/\Fq$ of genus~$h$ that
is a double cover of $C$ and that has at least as many rational points as
does~$C$.


\subsection{Odd characteristic}
\label{SS:odd}
In this section, we prove Theorem~\ref{T:explicit} for finite fields of odd 
characteristic. We start by stating and proving several lemmas that we will use
in the proof. 

\begin{lemma}
\label{L:2g1}
Let $q$ be an odd prime power and let $C/\Fq$ be a hyperelliptic curve of 
genus~$g$ 
with fewer than $q$ rational Weierstrass points. Then there is a 
hyperelliptic curve $D$ of genus $2g+1$ that is a double cover of $C$ and that
has at least as many rational points as does $C$.

If $C$ has exactly two rational Weierstrass points, then $D$ can be chosen to
have exactly two rational Weierstrass points.
\end{lemma}

\begin{remark}
\label{R:2g1}
If $C/\Fq$ is a hyperelliptic curve with $\#C(\Fq) > q + 3$ then $C$ has fewer 
than $q$ rational Weierstrass points.
\end{remark}

\begin{proof}
Let $\varphi$ be the canonical map from $C$ to $\PP^1$. There are at least $2$ 
points of $\PP^1$ that do not ramify in $\varphi$, and we can pick two such
points and choose a coordinate function $x$ on $\PP^1$ so that those two points
lie at $0$ and~$\infty$. That means that $C$ has a hyperelliptic model of the 
form $y^2 = f$, where $f\in \Fq[x]$ is a separable polynomial of degree $2g+2$
such that $f(0)\ne 0$.

Let $n$ be a nonsquare in $\Fq$, and consider the two hyperelliptic curves
$D\colon y^2 = f(x^2)$ and $D'\colon y^2 = f(nx^2)$. We note that both $f(x^2)$
and $f(nx^2)$ are separable polynomials of degree~$4g+4$, so both $D$ and $D'$
have genus $2g+1$. The two natural double covers $D\to C$ and $D'\to C$ are 
quadratic twists of one another, and it follows that 
$\#D(\Fq) + \#D'(\Fq) = 2\#C(\Fq)$. Therefore, one of these two curves has at
least as many rational points as does $C$.

Suppose $C$ has exactly two rational Weierstrass points. We can choose our
coordinate function $x$ on $\PP^1$ so that these two points lie over $x = 1$ and
$x = n$, where $n$ is a nonsquare element of $\Fq$, and proceed as before. Note 
that the Weierstrass point $(1,0)$ of $C$ splits into two rational Weierstrass 
points of $D$ and $(n,0)$ splits into nonrational points of $D$, while $(n,0)$ 
splits into two rational Weierstrass points of $D'$ and $(1,0)$ splits into
nonrational points of~$D'$. Therefore, each of $D$ and $D'$ has exactly two
rational Weierstrass points, so no matter which one we choose as our cover, we 
have the desired number of rational Weierstrass points.
\end{proof}

\begin{lemma}
\label{L:2g}
Let $q$ be an odd prime power and let $C/\Fq$ be a hyperelliptic curve of
genus~$g$ with exactly two rational Weierstrass points. Then there is a
hyperelliptic curve $D/\Fq$ of genus $2g$ that is a double cover of~$C$, that 
has exactly two rational Weierstrass points, and that has at least as many 
rational points as does $C$.
\end{lemma}

\begin{proof}[Proof of Lemma~\textup{\ref{L:2g}}]
Let $\varphi$ be the canonical map from $C$ to $\PP^1$, and choose a coordinate
function $x$ on $\PP^1$ so that the two rational Weierstrass points of $C$ lie 
over $x=0$ and $x=\infty$. Then $C$ has a hyperelliptic model of the form 
$y^2 = f$, where $f\in \Fq[x]$ is a separable polynomial of degree $2g+1$ such
that $f(0) = 0$ and such that $f$ has no other rational roots. By scaling $x$ 
and~$y$, if necessary, we can also assume that $f$ is monic. For every nonzero
$a\in\Fq$ let $h_a(x) = f(x^2 - a^2)$, so that $h_a$ is a monic separable 
polynomial in $\Fq[x]$ of degree $4g+2$ whose only rational roots are $x = a$
and $x = -a$. Let $D_a$ be the hyperelliptic curve $y^2 = h_a$. Then $D_a$ has
genus~$2g$ and is a double cover of $C$, and $D_a$ has exactly two rational
Weierstrass points. We will show that there is a value of $a$ so that 
$\#D_a(\Fq)\ge\#C(\Fq).$

Let $\chi$ denote the quadratic character on $\Fq$, so that for $z\in\Fq$ we
have $\chi(z) = 1$ if $z$ is a nonzero square, $\chi(z) = 0$ if $z=0$, and 
$\chi(z) = -1$ if $z$ is a nonsquare. Consider the degree-$4$ map $D_a\to P^1$
that takes a point $(x,y)$ of $D_a$ to the point $x^2 - a^2$ of $\PP^1$. A 
finite point $z$ of $\PP^1$ will have a rational point of $D$ lying over it if 
and only if there are rational solutions to $x^2 = z + a^2$ and the value of 
$f(z)$ is a square. Even stronger: The number of points of $D$ lying over $z$ is
equal to $(1 + \chi(z+a^2))(1 + \chi(f(z))$. On the other hand, if $z=\infty$
then there are two rational points of $D$ lying over it, because $h_a$ is monic
of even degree. Thus the number of rational points on $D$ is given by 
\begin{align*}
\#D_a(\Fq) 
&= 2 + \sum_{z\in \Fq} (1 + \chi(z+a^2))(1 + \chi(f(z))\\
&= 2 + \sum_{z\in \Fq} (1 + \chi(f(z))) 
     + \sum_{z\in \Fq} \chi(z+a^2) 
     + \sum_{z\in \Fq} \chi(z+a^2)\chi(f(z))\\
&= 1 + \#C(\Fq) + \sum_{z\in \Fq} \chi(z+a^2) 
                + \sum_{z\in \Fq} \chi(z+a^2)\chi(f(z))\\
&= 1 + \#C(\Fq) + \sum_{z\in \Fq} \chi(z+a^2)\chi(f(z))\,,
\end{align*}
where the third equality follows from the fact that 
$\#C(\Fq) = 1  + \sum_{z\in \Fq} (1 + \chi(f(z)))$ and the final equality
follows from the fact that the sum over all $z$ of $\chi(z+a^2)$ is zero, since
the number of nonzero squares in $\Fq$ is equal to the number of nonsquares.
Define
\[
N_a \colonequals \sum_{z\in \Fq} \chi(z+a^2)\chi(f(z))\,,
\]
so that $\#D_a(\Fq) = 1 + \#C(\Fq) + N_a$. To complete the proof, we need only
show that there is a nonzero $a$ such that $N_a \ge -1$.

We will also need an analog of $N_a$ when $a = 0$, which we define as follows. 
Let $c$ be the coefficient of $x$ in the polynomial $f$, and let $h_0$ be the 
monic separable polynomial of degree $4g$ such that $f(x^2) = x^2 h_0(x)$. Let
$D_0$ be the hyperelliptic curve of genus $2g-1$ given by $y^2 = h_0$. Arguing
as before, we find that 
\begin{align*}
\#D_0(\Fq) 
&= 2 + (1 + \chi(c)) + \sum_{z\in \Fq^\times} (1 + \chi(z))(1 + \chi(f(z))\\
&= 3 + \chi(c) + \sum_{z\in \Fq^\times} (1 + \chi(f(z))) 
               + \sum_{z\in \Fq^\times} \chi(z) 
               + \sum_{z\in \Fq^\times} \chi(z)\chi(f(z))\\
&= 1 + \chi(c) + \#C(\Fq) + \sum_{z\in \Fq^\times} \chi(z) 
                          + \sum_{z\in \Fq^\times} \chi(z)\chi(f(z))\\
&= 1 + \chi(c) + \#C(\Fq) + \sum_{z\in \Fq^\times} \chi(z)\chi(f(z))\\
&= 1 + \chi(c) + \#C(\Fq) + \sum_{z\in \Fq} \chi(z)\chi(f(z))\,.\\
\end{align*}
If we define $N_0$ to be
\[
N_0 \colonequals \sum_{z\in \Fq} \chi(z)\chi(f(z))\,,
\]
then $\#D_0(\Fq) = 1 + \chi(c) + \#C(\Fq) + N_0$. Since $D_0$ is hyperelliptic,
it can have at most $2q$ rational points in addition to the $1 + \chi(c)$ points
it has that lie over $x=0$. Thus, $N_0\le 2q - \#C(\Fq).$

Consider the sum, over all $a\in\Fq$, of $N_a$. We have
\begin{align*}
\sum_{a\in\Fq} N_a 
&= \sum_{a\in \Fq} \sum_{z \in \Fq} \chi(z+a^2)\chi(f(z))\\
&= \sum_{z\in\Fq} \chi(f(z)) \sum_{a\in\Fq} \chi(z+a^2)\\
&= \sum_{z\in\Fq^\times} \chi(f(z)) \sum_{a\in\Fq} \chi(z+a^2)
\end{align*}
where the last equality follows because $\chi(f(0)) = 0$. For nonzero $z\in\Fq$,
consider the genus-$0$ curve $X_z$ defined by $y^2 = x^2 + z$. The curve $X_z$
has two points at infinity, so arguing as before we find that 
\begin{align*}
\#X_z(\Fq) 
&= 2 + \sum_{a\in\Fq} (1 + \chi(a^2 + z))\\
&= q + 2 + \sum_{a\in\Fq} \chi(a^2 + z)\,.
\end{align*}
Since $X_z$ has genus~$0$ and hence $1 + q$ rational points, we see that
$\sum_{a\in\Fq} \chi(a^2 + z) = -1$ when $z\ne 0$. Therefore,
\[
\sum_{a\in\Fq} N_a 
= - \sum_{z\in\Fq^\times} \chi(f(z)) 
= - \sum_{z\in\Fq} \chi(f(z))  
= q -\sum_{z\in\Fq} (1 + \chi(f(z)))
= 1 + q - \#C(\Fq)\,.
\]

Suppose there were no nonzero $a$ with $N_a\ge -1$. Then we would have
\[
1 + q - \#C(\Fq) = \sum_{a\in\Fq} N_a  = N_0 + \sum_{a\in\Fq^\times} N_a
\le 2q - \#C(\Fq) + (q-1)(-2) = 2 - \#C(\Fq)\,,
\]
which would imply $1 + q\le 2$, a contradiction. Therefore there must be a 
nonzero $a$ with $N_a\ge -1$, and for every such $a$ the curve $D_a$ satisfies
the desired conditions.
\end{proof}

Lemmas~\ref{L:2g1} and~\ref{L:2g} give us the means to iterate a construction,
but we still need base curves to start with. These will be provided by
Lemmas~\ref{L:genus2} and~\ref{L:genus3}. To prepare for the proofs of those 
lemmas, we need some background information on $2$-isogenies and $2$-isogeny 
volcanoes.

Let $q$ be an odd prime power and let $t$ be an integer, coprime to~$q$, with
$t^2<4q$. Let $\Cc_t$ be the isogeny class of ordinary elliptic curves over $\Fq$
with trace $t$, and set $\Delta\colonequals t^2 - 4q < 0$. Write 
$\Delta = F^2\Delta_0$ for a fundamental discriminant $\Delta_0$. For every
divisor $f$ of $F$, there are elliptic curves in $\Cc_t$ whose endomorphism rings
are isomorphic to the quadratic order of 
discriminant~$f^2\Delta_0$~\cite[Theorem~4.2, pp.~538--539]{waterhouse}.

The \emph{height} of the isogeny class $\Cc_t$ (more properly, the height of the 
$2$-isogeny volcano associated to $\Cc_t$) is equal to the $2$-adic valuation of
the conductor~$F$. If $E$ is an elliptic curve in $\Cc_t$ whose endomorphism ring
has discriminant $f^2\Delta_0$, then the \emph{level} of $E$ is the $2$-adic 
valuation of~$f$. (This is the terminology of~\cite{FouquetMorain2002}; Kohel 
used different terminology in his thesis~\cite{kohel-phd}, which introduced 
these concepts.) We see that $\Cc_t$ contains elliptic curves of every level from
$0$ to the height of~$\Cc_t$.

\begin{lemma}
\label{L:volcanoes}
Let $\Cc_t$ be an ordinary isogeny class of trace $t$ and discriminant 
$\Delta= F^2\Delta_0$ as above, and suppose $\Cc_t$ has height $h>0$.
Then\textup{:}
\begin{itemize}
\item Every elliptic curve in $\Cc_t$ of level $h$ has exactly one rational
      point of order $2$, and the image of the corresponding $2$-isogeny is an
      elliptic curve of level~$h-1$.
\item Every elliptic curve in $\Cc_t$ of level $\ell$ with $0 < \ell < h$ has 
      exactly three rational points of order~$2$. For two of these points, the 
      image of the corresponding $2$-isogeny is a curve of level $\ell+1$, and
      for the third the image is an elliptic curve of level $\ell-1$. 
\item Every elliptic curve in $\Cc_t$ of level $0$ has exactly three rational
      points of order~$2$. The images of the corresponding $2$-isogenies are 
      curves of level $0$ or $1$, and the number of point giving rise to a curve
      of level $0$ is equal to $2$, $1$, or $0$, corresponding to whether 
      $\Delta_0\equiv 1\bmod 8$, $\Delta_0\equiv0 \bmod 4$, or
      $\Delta_0\equiv 5\bmod 8$.
\end{itemize}
\end{lemma}

\begin{proof}
This follows immediately from~\cite[Theorem~2.1, p.~278]{FouquetMorain2002} 
or~\cite[Prop.~23, p.~54]{kohel-phd}.
\end{proof}

\begin{lemma}
\label{L:2torsion}
Let $q$ be an odd prime power and let $E\colon y^2 = (x-a)(x-b)(x-c)$ be an
elliptic curve over $\Fq$ with all of its $2$-torsion rational. Let 
$\varphi\colon E\to F$ be the $2$-isogeny with kernel generated by the point 
$(a,0)$ on $E$. Then all of the $2$-torsion of $F$ is rational if and only if
$(a-b)(a-c)$ is a square.
\end{lemma}

\begin{proof}
By shifting $x$-coordinates on $E$ by $a$, we see that it suffices to prove the
statement when $a = 0$. Using~\cite[Example~4.5, p.~70]{Silverman:AEC}, for 
example, we find that one model for $F$ is given by
$y^2 = x^3 + 2(b+c)x^2 + (b-c)^2x$. The $2$-torsion of $F$ is all rational if
and only if the quadratic $x^2 + 2(b+c)x + (b-c)^2$ splits, which happens if and
only if its discriminant $16bc$ is a square.
\end{proof}

\begin{lemma}
\label{L:genus2}
Let $q$ be an odd prime power. Then there is a curve $C/\Fq$ of genus~$2$ with 
exactly two rational Weierstrass points such that
\[
\#C(\Fq) >  \begin{cases}
            1 + q + 4\sqrt{q} -  5 & \text{if $q<512$};\\
            1 + q + 4\sqrt{q} - 32 & \text{if $q>512$}.\\
            \end{cases}
\]
\end{lemma}

\begin{proof}
For $q<512$ we find examples of such curves by computer search; a list of 
examples is included with the ancillary files included with the arXiv version
of this paper. Thus we may assume that $q > 512$.

First consider the case where $q\equiv 3\bmod 4$. Let $t$ be the largest integer
such that $t\equiv q + 1\bmod 8$ and $(t,q) = 1$ and $t^2 \le 4q$, so that 
$t>2\sqrt{q} - 16>0$. Set $\Delta \colonequals t^2 - 4q$ and write 
$\Delta = F^2\Delta_0$ for a fundamental discriminant $\Delta_0$. We note that
$\Delta \equiv 4\bmod 32$, so that $F\equiv 2\bmod 4$ and 
$\Delta_0\equiv 1\bmod 8$. If we let $\Cc_{-t}$ be the isogeny class of elliptic
curves of trace $-t$ over~$\Fq$, then $\Cc_{-t}$ has height~$1$.

Let $E$ be an elliptic curve in $\Cc_{-t}$ of level~$0$. By Lemma~\ref{L:volcanoes},
$E$ has two rational $2$-isogenies to elliptic curves of level~$0$ and one
rational $2$-isogeny to an elliptic curve $E'$ of level~$1$. Write $E$ as 
$y^2 = x(x-a)(x-b)$, with coordinates chosen so that the $2$-isogeny from $E$ to
$E'$ has kernel generated by $(0,0)$. Since $E'$ does not have all of its 
$2$-torsion rational, Lemma~\ref{L:2torsion} shows that $ab$ is not a square; on
the other hand, since the other curves $2$-isogenous to $E$ \emph{do} have their
$2$-torsion rational, we find that $a(a-b)$ and $b(b-a)$ are both squares.

Define elements $\alpha_i$ and $\beta_i$ of $\Fq$, for $i=1,2,3$, by
\[
\alpha_1\colonequals 0,\quad
\alpha_2\colonequals a,\quad
\alpha_3\colonequals b,\quad
\beta_1\colonequals a,\quad
\beta_2\colonequals b,\quad
\beta_3\colonequals 0.
\]
Let $\psi\colon E[2]\to E[2]$ be the isomorphism that takes $(\alpha_i,0)$ to 
$(\beta_i,0)$ for each $i$. Since the discriminant of the endomorphism ring of 
$E$ is congruent to $1\bmod 8$, the curve $E$ has no nontrivial automorphisms,
and $\psi$ is not the restriction to $E[2]$ of an automorphism of $E$. 
Therefore, from \cite[Props.~3 and~4, p.~324]{howe00} we obtain a curve $C$ of
genus~$2$ whose Jacobian is isogenous to $E^2$, and the formulas in the cited
results give us a model for $C$. In this case, we find that $C$ is given by 
$y^2 = h$ with
\[
h = a^5 b^5 (a-b)^5 (a^2 - ab + b^2)^3 
     \Bigl(x^2 + \frac{b}{a}\Bigr) 
     \Bigl(x^2 - \frac{a-b}{b}\Bigr) 
     \Bigl(x^2 - \frac{a}{b-a}\Bigr)\,.
\]
Now, since $ab$ and $-1$ are both nonsquare, it follows that $-b/a$ is a square, 
so the first quadratic factor in $h$ splits. On the other hand, since $b(b-a)$ 
is a square, we see that $(a-b)/b$ is \emph{not} a square, so the second
quadratic factor of $h$ is irreducible. Likewise, the third quadratic factor is
irreducible. Therefore, $C$ has exactly two rational Weierstrass points.

Since the Jacobian of $C$ is isogenous to $E^2$, we have 
$\#C(\Fq) = 1 + q + 2t > 1 + q + 4\sqrt{q} - 32.$

Now we turn to the case where $q \equiv 1 \bmod 4$. Let $t$ be the largest
integer such that $t\equiv 2\bmod 4$ and $(t,q) = 1$ and $t^2 \le 4q$, so that
$t>2\sqrt{q} - 8 > 0$. Let $\Cc_{-t}$ be the isogeny class of ordinary elliptic 
curves over $\Fq$ with trace~$-t$, set $\Delta \colonequals t^2 - 4q$, and write
$\Delta = F^2\Delta_0$ for a fundamental discriminant $\Delta_0$. We note that
$\Delta \equiv 0\bmod 16$, so that either the height $h$ of $\Cc_{-t}$ is at 
least~$2$, or $h=1$ and $\Delta_0 \equiv 0\bmod 4$.

Let $E$ be an elliptic curve in $\Cc_{-t}$ of level~$h-1$. Lemma~\ref{L:volcanoes}
shows that if $h\ge 2$, then $E$ has two $2$-isogenies to elliptic curves of
level $h$ and one to an elliptic curve of level~$h-2$. The curves of level $h$
have only one rational point of order~$2$, while the curve of level $h-2$ has 
three rational points of order~$2$. On the other hand, if $h=1$ and 
$\Delta_0\equiv0\bmod 4$ then $E$ has two $2$-isogenies to elliptic curves of
level $1$ and one to an elliptic curve of level~$0$. Again, the curves of level
$1$ have only one rational point of order~$2$, while the curve of level $0$ has
three rational points of order~$2$.

We see that in every case, $E$ can be written in the form $y^2 = x(x-a)(x-b)$
where $ab$ is a square and where $(a-b)/a$ and $(b-a)/b$ are both nonsquares.
Taking $\alpha_i$ and $\beta_i$ as before, we find that the curve $C$ given by
$y^2 = h$, with
\[
h = a^5 b^5 (a-b)^5 (a^2 - ab + b^2)^3 
     \Bigl(x^2 + \frac{b}{a}\Bigr) 
     \Bigl(x^2 - \frac{a-b}{b}\Bigr) 
     \Bigl(x^2 - \frac{a}{b-a}\Bigr)\,,
\]
has Jacobian isogenous to $E^2$. We again find that the first quadratic factor
splits and the other two are irreducible, so once again $C$ has exactly two 
rational Weierstrass points. We also once again have $\#C(\Fq) = 1 + q + 2t.$

Thus, for every odd prime power $q$ we have shown that there is a curve of
genus~$2$ over $\Fq$ with exactly two rational Weierstrass points and with
$\#C(\Fq) > 1 + q + 4\sqrt{q} - 32.$
\end{proof}

\begin{lemma}
\label{L:genus3}
Let $q$ be an odd prime power. Then there is a hyperelliptic curve $C/\Fq$ of
genus~$3$ with exactly two rational Weierstrass points and with
\[
\#C(\Fq) >  \begin{cases}
            1 + q + 4\sqrt{q} -  5 & \text{if $q<512$};\\
            1 + q + 4\sqrt{q} - 32 & \text{if $q>512$}.\\
            \end{cases}
\]
\end{lemma}

\begin{proof}
The statement for $q<512$ is verified by computer search; a list of examples of
such curves can be found in the ancillary files included with the arXiv version
of this paper. We assume now that $q>512$.

First consider the case where $q\equiv 3\bmod 4$. As in the proof of 
Lemma~\ref{L:genus2}, we let $t$ be the largest integer such that 
$t\equiv q + 1\bmod 8$ and $(t,q) = 1$ and $t^2 \le 4q$, so that 
$t>2\sqrt{q} - 16>0$. We see that $\Delta\colonequals t^2 - 4q$ can be written 
$F^2\Delta_0$ for a fundamental discriminant $\Delta_0$ that is congruent to $1$
modulo $8$ and a conductor $F$ that is congruent to $2$ modulo~$4$. If we let
$\Cc_{-t}$ be the isogeny class of elliptic curves of trace $-t$ over~$\Fq$, then 
$\Cc_{-t}$ has height~$1$.

Let $E$ be an elliptic curve in $\Cc_{-t}$ of level~$0$. Lemma~\ref{L:volcanoes}
shows that $E$ has two rational $2$-isogenies to elliptic curves of level~$0$.
Let $E'$ be one of these curves, and write $E$ as $y^2 = x(x-a)(x-b)$, with
coordinates chosen so that the $2$-isogeny from $E$ to $E'$ has kernel 
generated by $(0,0)$. Since $E'$ is of level $0$ it has all of its $2$-torsion
rational, so Lemma~\ref{L:2torsion} shows that $ab$ is a square; therefore $b/a$
is also a square, say $b/a = c^2$.

Define elements $\alpha_i$ and $\beta_i$ of $\Fq$, for $i=1,2,3$, by
\[
\alpha_1\colonequals 0,\quad
\alpha_2\colonequals a,\quad
\alpha_3\colonequals b,\quad
\beta_1\colonequals 0,\quad
\beta_2\colonequals b,\quad
\beta_3\colonequals a,
\]
and let $\psi\colon E[2]\to E[2]$ be the isomorphism that takes $(\alpha_i,0)$ 
to $(\beta_i,0)$ for each $i$. Since the discriminant of the endomorphism ring
of $E$ is congruent to $1\bmod 8$, the curve $E$ has no nontrivial 
automorphisms, and $\psi$ is not the restriction to $E[2]$ of an automorphism 
of~$E$. Once again we use~\cite[Props.~3 and~4, p.~324]{howe00} to show that
there is a genus-$2$ curve $C$ whose Jacobian is $(2,2)$-isogenous to $E^2$, and
we find that one model for such a $C$ is given by $y^2 = h$ with
\[
h = a^5 b^5 (a-b)^8 (a + b)^3 
     \Bigl(x^2 - c^2\Bigr) 
     \Bigl(x^2 - 1/c^2\Bigr)
     \Bigl(x^2 + 1\Bigr)\,.
\]
Since the Jacobian of $C$ is isogenous to $E^2$, we have 
$\#C(\Fq) = 1 + q + 2t > 1 + q + 4\sqrt{q} - 32.$

Now we replace $x$ with $(2c^2x + 1-c^2)/(2cx + c^3 - c)$ and scale $y$
appropriately to find that $C$ can also be written $y^2 = f$, where
\[
f = -a x (x - 1) \Bigl(x + \frac{(c^2-1)^2}{4c^2}\Bigr)
                \Bigl(x^2 + \frac{(c^2-1)^2}{4c^2}\Bigr)\,.
\]
Note that the quadratic factor is irreducible.

Now, as in the proof of Lemma~\ref{L:2g1}, we consider two double covers of~$C$.
Let $g = f/x$, let $n$ be a nonsquare element of $\Fq$, let $D$ be the curve 
$y^2 = g(x^2)$, and let $D'$ be the curve $y^2 = g(nx^2)$. The curves $D$ and
$D'$ are both double covers of $C$: The map $(x,y) \mapsto (x^2,xy)$ sends $D$
to $C$, and the map $(x,y) \mapsto (nx^2,xy)$ sends $D'$ to $C$. The two covers
are quadratic twists of one another, so we have 
$\#D(\Fq) + \#D'(\Fq) = 2\#C(\Fq)$, so at least one of $D$ and $D'$ has at least
as many points as $C$. Also, the two curves are  both hyperelliptic of 
genus~$3$. Furthermore, the rational Weierstrass point $(1,0)$ of $C$ splits 
into two rational Weierstrass points of $D$ while the rational Weierstrass point
lying over $x = -(c^2 - 1)^2/(4c^2)$ does not, because $-1$ is not a square, 
whereas for $D'$ the splitting behavior of these two points is reversed. Thus,
both $D$ and $D'$ have exact two rational Weierstrass points. 

We have shown that there is a hyperelliptic curve over $\Fq$ of genus~$3$ with 
exactly two rational Weierstrass points and with more than 
$1 + q + 4\sqrt{q} - 32$ rational points.

Now consider the case $q\equiv 1\bmod 4$. If $q\equiv 1\bmod 16$ or
$q\equiv 13\bmod 16$, let $t$ be the largest integer congruent to $2$ or $14$
modulo $16$ that is coprime to $q$ and such that $t^2<4q$; if
$q\equiv 5\bmod 16$ or $q\equiv 9\bmod 16$, let $t$ be the largest integer
congruent to $6$ or $10$ modulo $16$ that is coprime to $q$ and such that 
$t^2<4q$. In both cases we have $t>2\sqrt{q}-16$.

Let $\Delta \colonequals t^2 - 4q$ and write $\Delta = F^2\Delta_0$ for a
fundamental discriminant $\Delta_0$. Our choice of $t$ guarantees that 
$\Delta\equiv 0\bmod 64$, so $F\equiv 0\bmod 4$. If we let $\Cc_{-t}$ be the isogeny
class of elliptic curves of trace $-t$ over~$\Fq$, then the height $h$ of $\Cc_{-t}$ 
is at least~$2$.

Let $E$ be an elliptic curve in $\Cc_{-t}$ of level $h-2$, and let $E'$ be an 
elliptic curve of height $h-1$ that is $2$-isogenous to $E$. Write $E$
as $y^2 = x(x-a)(x-b)$ so that the kernel of the $2$-isogeny to $E'$ contains
the point $(0,0)$. Since $E'$ has all of its $2$-torsion rational,
Lemma~\ref{L:2torsion} shows that $ab$ is a square, so we can write 
$b = ac^2$ for some $c\in\Fq$. Then the fact that the other two curves 
$2$-isogenous to $E$ have all of their $2$-torsion rational implies
that $a^2(1-c^2)$ and $a^2c^2(c^2-1)$ are squares, so $c^2 - 1$ is a square.

We compute that one model for $E'$ is given by
\[
y^2 = x \bigl(x + a(c+1)^2\bigr) \bigl(x+a(c-1)^2\bigr)\,.
\]
Here, the isogeny $E'\to E$ corresponds to the $2$-torsion point $(0,0)$.
The other two $2$-isogenous take $E'$ to curves of level $h$, which each
have exactly one rational point of order~$2$, so Lemma~\ref{L:2torsion} 
tells us that the two values
$4a^2c(c+1)^2$ and $-4a^2c(c-1)^2$ are not squares.
This shows that $c$ is not a square.

Now we use the formulas from~\cite[Prop.~4, p.~324]{howe00} to construct
a curve of genus~$2$ whose Jacobian is $(2,2)$-isogenous to $E\times E'$. In 
particular, we take
\[
\alpha_1\colonequals 0,\quad
\alpha_2\colonequals a,\quad
\alpha_3\colonequals ac^2,\quad
\beta_1\colonequals 0,\quad
\beta_2\colonequals -a(c+1)^2,\quad
\beta_3\colonequals -a(c-1)^2
\]
in~\cite[Prop.~4, p.~324]{howe00} and we find that the resulting curve
$C$ is given by $y^2 = h$, where
\[
h = 64 c^7 (c^2-1)^8 (c^2+1)^3 (c^2+2 c-1)^3 a^{21}
     \Bigl(x^2 - \frac{c^2-1}{4c}\Bigr) 
     \Bigl(x^2 + \frac{1}{(c+1)^2}\Bigr) 
     \Bigl(x^2 + \frac{c^2}{(c-1)^2}\Bigr)\,.
\]
Since the Jacobian of $C$ is isogenous to $E^2$, we have 
$\#C(\Fq) = 1 + q + 2t > 1 + q + 4\sqrt{q} - 32.$

The first of the quadratic factors in the above expression for $h$
is irreducible, because $c^2-1$ is a square
but $c$ is not. The other two quadratic factors split, and if we let $i$
denote a square root of $-1$ in $\Fq$, then the roots of $h$ are
\[
\frac{i}{c+1}\,, \qquad
\frac{-i}{c+1}\,, \qquad
\frac{ic}{c-1}\,, \quad\text{and}\quad
\frac{-ic}{c-1}\,.
\]
If we replace $x$ with 
\[
\frac{
i(c^2 + 2c - 2)x - 2ic(c + 1)
}{
(c + 1)(c^2 + 2c - 1)x - 2(c^2 - 1)
}
\]
and rescale $y$ appropriately, we find that $C$ can also be written
as $y^2 = f$, where
\[
f = a x (x - 1)
\biggl(x - \frac{4 c (c^2 - 1)}{(c^2 + 2c - 1)^2}\biggr)
\biggl(x^2 - \frac{4 (c+1)(c^2 + 1)}{(c^2 + 2c - 1)^2}x  + \frac{4(c+1)^2}{(c^2 + 2c - 1)^2}\biggr)\,.
\]
Note that the quadratic factor is irreducible, and that the root 
$r\colonequals 4 c (c^2 - 1)/(c^2 + 2c - 1)^2$ is not a square, because $c$ is not a square while
$c^2 - 1$ is.

Let $g = f/x$, let $n$ be a nonsquare element of $\Fq$, let $D$ be the
curve $y^2 = g(x^2)$, and let $D'$ be the curve $y^2 = g(nx^2)$. 
As before, the  curves $D$ and $D'$ are both double covers of $C$
and the two covers are quadratic twists of one another, so 
at least one of $D$ and $D'$ has at least as many points as $C$.
The two curves are both hyperelliptic of genus~$3$. And finally,
the rational Weierstrass point $(1,0)$ of $C$ splits into two rational Weierstrass
points of $D$ while the rational Weierstrass point $(r,0)$ does not, 
whereas for $D'$ the splitting behavior of these
two points is reversed. Thus, both $D$ and $D'$ have exact two rational Weierstrass
points. 

We have shown that there is a hyperelliptic curve of genus~$3$ over $\Fq$ with 
exactly two rational Weierstrass points and with
more than $1 + q + 4\sqrt{q} - 32$ rational points.
\end{proof}

With all of these preparatory results at hand, the proof of Theorem~\ref{T:explicit}
for odd $q$ is very short.

\begin{proof}[Proof of Theorem~\textup{\ref{T:explicit}(i)}]
We prove the result by induction on $g$. The statement is true for $g=2$ and $g=3$
by Lemmas~\ref{L:genus2} and~\ref{L:genus3}. Now suppose Theorem~\ref{T:explicit}(i)
holds for all $g$ less than some integer $G>3$. We will show it also holds 
when $g = G$. 

For convenience's sake, let us set $c_q = 5$ if $q<512$ and $c_q=32$ if $q>512$.
Set $h = \lfloor G/2\rfloor$. Then $h>1$, and we may apply 
Theorem~\ref{T:explicit}(i) to show that for every $q$, there is a hyperelliptic
curve $C/\Fq$ of genus~$h$ with exactly two rational Weierstrass points and with
$\#C(\Fq) > 1 + q + 4\sqrt{q} - c_q$.

If $G$ is odd, we apply Lemma~\ref{L:2g1} to the curve $C$ and find that there
is a hyperelliptic curve $D$ of genus $2h+1 = G$ with exactly two rational
Weierstrass points and with $\#D(\Fq)\ge \#C(\Fq)$.

If $G$ is even, we apply Lemma~\ref{L:2g} to the curve $C$ and find that there 
is a hyperelliptic curve $D$ of genus $2h = G$ with exactly two rational
Weierstrass points and with $\#D(\Fq)\ge \#C(\Fq)$.
\end{proof}

\subsection{Characteristic \texorpdfstring{$2$}{2}}
\label{SS:2}

In this section, we prove Theorem~\ref{T:explicit} for finite fields of
characteristic~$2$. The spirit of the proof is very similar to the odd
characteristic case,
but the switch to Artin--Schreier extensions instead of Kummer extensions
requires a few technical modifications. We begin with some basic observations
about hyperelliptic curves in characteristic~$2$.

If $q$ is a power of $2$ and $C$ is a hyperelliptic curve over $\Fq$, then
$C$ has a model of the form $y^2 + y = f$ for a rational function $f\in\Fq(x)$.
Replacing $y$ with $y + u$ for a rational function $u\in\Fq(x)$ turns the
equation for $C$ into $y^2 + y = f + u^2 + u$,  and by modifying $f$ in this way
we can assume that all of the poles of $f$, including the pole at $\infty$, have
odd order. Suppose $f$ has $r$ poles, with orders $d_1,\ldots,d_r$. Then the 
Weierstrass points of $C$ are precisely the points lying over the poles of $f$
in $\PP^1$, and~\cite[Prop.~3.7.8, p.127]{Stichtenoth2009} shows that the genus 
$g$ of $C$ is given by
\[
g = - 1 + \sum_{i=1}^r \frac{d_i + 1}{2}\,.
\]

\begin{lemma}
\label{L:genus}
Let $q$ be a power of $2$ and let $C/\Fq$ be a hyperelliptic curve of genus~$g$,
so that $C$ has a model $y^2 + y = f$ for a rational function $f\in\Fq(x)$
all of whose poles have odd order. Given $a\in\Fq^\times$ and $b\in\Fq$, let 
$h = f((x^2+x+b)/a)$ and let $D$ be the curve defined by $y^2 + y = h$.
Then\textup{:}
\begin{itemize}
\item If $f$ has no pole at infinity, the genus of $D$ is $2g+1$.
\item If $f$ has a pole of order $d>1$ at infinity, the genus of $D$ is $2g$.
\item If $f$ has a simple pole at infinity, write $f = cx + F$ for a 
      constant $c\in\Fq^\times$ and a rational
      function $F$ with no pole at infinity. Then the genus of $D$ is
      $2g$ if $a\ne c$, and is $2g-1$ if $a = c$.
\end{itemize}
\end{lemma}

\begin{proof}
Each finite pole of $f$, say of order $d$, gives rise to two finite poles of $h$, 
each of order $d$, and if $f$ has no pole at infinity then neither does $h$.
If $f$ has a total of $r$ poles, none of them at infinity,
of orders $d_1,\ldots, d_r$, then $h$ has $2r$ poles,
of orders $d_1, d_1, d_2, d_2, \ldots, d_r, d_r$, and the genus of $D$ is
given by
\[
-1 + 2\sum_{i=1}^r \frac{d_i+1}{2} 
= 1 + 2\Bigl(-1 + 2\sum_{i=1}^r \frac{d_i+1}{2}\Bigr)
= 1 + 2g\,.
\]

Suppose $f$ has a pole at infinity of order $d$. We can write
\begin{equation}
\label{EQ:polar}
f = c_d x^d + c_{d-1} x^{d-1} + \ldots + c_1 x + F,,
\end{equation}
where the $c_i$ are elements in $\Fq$ with $c_d\ne0$ and
$F$ is a rational function with no pole at infinity.
Then the polar decomposition of $h$ at infinity is
\[
h_a = (c_d/a^d) x^{2d} + (d c_d / a^d) x^{2d-1} + (\text{lower degree terms})\,.
\]
If $d>1$ we can replace $y$ with $y + (\sqrt{c_d/a^d})x^d$ to find that at infinity
the curve $D$ looks like 
\[
y^2 + y = (d c_d / a^d) x^{2d-1} + (\text{lower degree terms})\,.
\]
The contribution of the pole of $h$ at infinity to the genus of $D$ is $d$,
while the contribution of the pole of $f$ at infinity to the genus of $C$ is $(d+1)/2$.
Combining this with the contributions of the finite poles, which behave as above,
we find that the genus of $D$ is $2g$.

On the other hand, if $d=1$ then we have
\[
h_a = (c_1/a) x^{2} +  (c_1 / a) x  + (\text{lower degree terms})\,,
\]
and by replacing $y$ with $y + (\sqrt{c_1/a})x$ we find that at infinity
the curve $D_a$ looks like 
\[
y^2 + y = (c_1 / a +  \sqrt{c_1/a}) x+ (\text{lower degree terms})\,.
\]
If $a\ne c_1$ then the coefficient of $x$ is nonzero and the 
contribution of the pole at infinity to the genus is~$1$, and we compute as
above that the genus of $D$ is $2g$. On the other hand, 
if $a = c_1$ then the pole at infinity can be removed, and there is
no contribution to the genus.  In this case we check that that 
genus of $D$ is $2g-1$.
\end{proof}

\begin{lemma}
\label{L:2g1char2}
Let $q$ be power of $2$ and
let $C/\Fq$ be a hyperelliptic curve of genus~$g$ with fewer than $q+1$
rational Weierstrass points. Then there is a hyperelliptic curve $D$ of
genus $2g+1$ that is a double cover of $C$ and that has at least as many 
rational points as does $C$.

If in addition $C$ has at least two rational Weierstrass points,
then $D$ can be chosen to have at least two rational points.
\end{lemma}

\begin{remark}
\label{R:2g1char2}
If $C/\Fq$ is a hyperelliptic curve with $\#C(\Fq) > q + 1$ then
$C$ has fewer than $q+1$ rational Weierstrass points.
\end{remark}

\begin{proof}
Let $\varphi$ be the canonical map from $C$ to $\PP^1$. There is at least one point
of $\PP^1$ that does not ramify in $\varphi$, and we can pick one such point and
choose a coordinate function $x$ on $\PP^1$ so that this point lies at~$\infty$.
That means that $C$ has a hyperelliptic model of the form $y^2 + y = f$, where
$f\in\Fq(x)$ is a rational function, all of whose poles have odd order, and with
no pole at~$\infty$. 

Let $n$ be an element of $\Fq$ whose absolute trace --- that is, its trace 
to~$\F_2$ --- is equal to~$1$. Consider the two hyperelliptic curves
$D\colon y^2 + y = f(x^2 + x)$ and $D'\colon y^2 + y = f(x^2 + x + n)$. 
By Lemma~\ref{L:genus}, both $D$ and $D'$ have genus~$2g+1$.
Also, the obvious double covers $D\to C$ and $D'\to C$ are quadratic twists of one another,
and it follows that $\#D(\Fq) + \#D'(\Fq) = 2\#C(\Fq)$. Therefore,
one of these two curves has at least as many rational points as does $C$.

Suppose $C$ has at least two rational Weierstrass points. We can choose our
coordinate function $x$ so that one of these points lies over 
the point $x=0$ of $\PP^1$, which the other lies over the point $x=n$ of~$\PP^1$,
where $n$ is an element of $\Fq$ of absolute trace~$1$. The Weierstrass point
lying over $x=0$ splits into two rational Weierstrass points of~$D$, and the
Weierstrass point lying over $x=n$ splits into two rational Weierstrass points of~$D'$.
Therefore, no matter which of the two curves has at least as many rational
points as~$C$, it has at least two rational Weierstrass points.
\end{proof}

\begin{lemma}
\label{L:2gchar2}
Let $q>2$ be a power of $2$ and let $C/\Fq$ be a hyperelliptic curve of
genus~$g$ that has at least two rational Weierstrass points and
with $\#C(\Fq) > q$.
Then there is a hyperelliptic curve $D/\Fq$ of genus $2g$
that is a double cover of $C$, that has at least two rational Weierstrass points,
and that satisfies
\begin{alignat*}{2}
\#D(\Fq) &\ge \#C(\Fq)&\qquad&\text{if $\#C(\Fq) < 2q$}\textup{;}\\
\#D(\Fq) &= 2q-1      &      &\text{if $\#C(\Fq) = 2q$.}
\end{alignat*}
\end{lemma}

\begin{proof}
Let $\varphi\colon C\to\PP^1$ be the canonical double cover.
Suppose $C$ has at least three rational Weierstrass points. We can choose
a coordinate function $x$ on $\PP^1$ so that these three points lie over
$0$, $n$, and~$\infty$, where $n$ is an element of $\Fq$ of absolute trace~$1$.
If $f$ has a simple pole at infinity and if the coefficient $c_1$ of $x$
in the polar expansion~\eqref{EQ:polar} is equal to $1$, we can choose 
another element $n'$ of $\Fq$ of absolute trace $1$ and then scale $x$ by
a factor of $n'/n$; this has the effect of replacing $n$ with $n'$ and of
modifying the coefficient $c_1$ so that it is no longer equal to~$1$.
Thus, we may assume that $c_1\ne 1$.

Consider the two curves $D\colon y^2 + y = f(x^2 + x)$ and
$D'\colon y^2 + y = h(x^2 + x + n)$. We see from Lemma~\ref{L:genus} that
both $D$ and $D'$ have genus $2g$, and one of them has at least as many
rational points
as does~$C$. Furthermore, the Weierstrass point of $C$ that lies above $x=0$
splits into two rational Weierstrass points on~$D$, while the Weierstrass
point of $C$ lying above $(n,0)$ splits into two rational Weierstrass points on~$D'$.
Thus, at least one of $D$ and $D'$ will satisfy the conclusion of the lemma. 

We can now turn to the case where $C$ has exactly two rational Weierstrass points.
This time, we can  choose a coordinate function $x$ on $\PP^1$ so that 
$C$ has a hyperelliptic model of the form $y^2 + y = f$, where 
$f\in \Fq(x)$ is a rational function that has odd-order poles at $0$ and~$\infty$
and no other rational poles. If $\#C(\Fq) < 2q$, then there must be a rational point
of $\PP^1$ that does not have a rational point of $C$ lying over it; in this case,
we scale $x$ so that this point becomes the point $x=1$ of $\PP^1$. In particular,
this means that $f(1)$ has absolute trace~$1$. On the other hand, if $\#C(\Fq) = 2q$,
then every point of $\PP^1(\Fq)$ other than $0$ and $\infty$ splits into
two rational points of $C$, so $f(a)$ has absolute trace $0$ for every $a\in\Fq^\times.$

Let $n$ be an element of $\Fq$ whose absolute trace is~$1$. For every $a\in\Fq$,
set $h_a\colonequals f((x^2+x)/(a^2+a+n))$ and let $D_a$ be the curve
$y^2 + y = h_a$. If the polynomial $x^2 + x + n = c_1$ (with $c_1$ defined in \eqref{EQ:polar})  has roots in~$\Fq$,
we let $s_0$ and $s_1$ be those roots; otherwise, we
take $s_1 = 0$ and $s_2 = 1$. We see from Lemma~\ref{L:genus} that
$D_a$ has genus $2g$ for all $a\in\Fq$ different from $s_1$ and~$s_2$.

Let us compute the number of rational points on $D_a$.
Let $\chi\colon\PP^1(\Fq)\to \{-1,0,1\}$
be the function that takes $\infty$ to $0$ and that takes an element $z\in\Fq$ to
$1$ or $-1$ depending on whether the absolute trace of $z$ is $0$ or $1$. We see,
for example, that
\[
\#C(\Fq) = 1 + \sum_{z\in\Fq}(1 + \chi(f(z))) = 1 + q + \sum_{z\in\Fq} \chi(f(z))\,,
\]
just as we had in odd characteristic. Let $n_a$ be the number of rational points
on $D_a$ at infinity, so that $n_a = 1$ if $D_a$ has a Weierstrass point
at infinity and $n_a$ is either $0$ or $2$ otherwise.
Arguing as in the odd-characteristic case,  we find that
\begin{align*}
\#D_a(\Fq) 
&= n_a + \sum_{z\in\Fq}(1 + \chi((a^2+a+n)z))(1 + \chi(f(z)))\\
&= n_a - 1 + \#C(\Fq) + \sum_{z\in\Fq}\chi((a^2+a+n)z) +  \sum_{z\in\Fq} \chi((a^2+a+n)z)\chi(f(z))\\
&= n_a - 1 + \#C(\Fq) + \sum_{z\in\Fq} \chi((a^2+a+n)z)\chi(f(z))\,.
\end{align*}
Define
\[
N_a \colonequals \sum_{z\in\Fq} \chi((a^2+a+n)z)\chi(f(z))
\]
so that
\begin{equation}
\label{EQ:Narel}
\#D_a(\Fq) - n_a = \#C(\Fq) - 1 + N_a\,.
\end{equation}

Suppose $\#C(\Fq) = 2q$, so that $\chi(f(z)) = 1$ for all $z\in\Fq^\times$.
Then for every $a\in\Fq\setminus\{s_0,s_1\}$ we have
\[
N_a = \sum_{z\in\Fq} \chi((a^2+a+n)z)\chi(f(z)) = \sum_{z\in\Fq^\times} \chi((a^2+a+n)z) = -1\,,
\]
and for such $a$ we also know that $n_a=1$. Thus, for every such $a$ the curve
$D_a$ has $2q-1$ points, and so satisfies the conditions listed in the lemma.

We are left with the case where $\#C(\Fq)<2q$.  We will show that in this case
there is an $a\in\Fq$, not equal to $s_1$ or $s_2$, such that $N_a\ge 0$. 
Recall that in this case, we normalized $f$ so that $\chi(f(1)) = -1$.

We compute:
\begin{align*}
\sum_{a\in\Fq} N_{a} 
&= \sum_{a\in\Fq} \sum_{z\in\Fq} \chi((a^2 + a + n)z)\chi(f(z))\\
&= \sum_{z\in\Fq} \chi(f(z)) \sum_{a\in\Fq} \chi((a^2+a+n)z)\,.
\end{align*}
When $z = 0$, the interior sum is equal to $q$. 
When $z = 1$, the interior sum is equal to $-q$. 
When $z\in\Fq\setminus\F_2$, we calculate the interior sum as follows.
Let $X_z$ be the genus-$0$ curve $y^2 + y = z(x^2 + x + n)$.
Then 
\[
1 + q = \#X_z(\Fq) = 1 + \sum_{a \in\Fq} (1 + \chi((a^2+a+n)z)) 
    = 1 + q + \sum_{a \in\Fq} \chi((a^2+a+n)z)\,,
\]
so 
\[
\sum_{a\in\Fq} \chi((a^2+a+n)z) = 0\,.
\]
Therefore, 
\[
\sum_{a\in\Fq} N_a = q\chi(f(0)) - q \chi(f(1)) = q\cdot 0 - q\cdot(-1) = q\,.
\]
Now, from~\eqref{EQ:Narel} we find that $N_a = (\#D_a(\Fq)-n_a) - (\#C(\Fq)-1)$,
so that $N_a$ is the difference between the number of rational 
points of $D_a$ not lying over $\infty$ and the number of rational points
of $C$ not lying over~$\infty$. Since $D_a$ has two rational Weierstrass points
not lying over $\infty$ we have $\#D_a(\Fq)-n_a\le 2q-2$, so 
$N_a\le 2q-1-\#C(\Fq)$. Suppose, to obtain a contradiction, that for every
$a\in\Fq$ other than $s_1$ and $s_2$ we had $\#D_a(\Fq) < \#C(\Fq)$, so that 
$N_a\le -1$. Then we would have
\[
q = \sum_{a\in\Fq} N_a 
= N_{s_1} + N_{s_2} + \sum_{a\in\Fq\setminus\{s_1,s_2\}} N_a
\le 4q - 2 - 2\#C(\Fq) - (q-2) 
= 3q - 2\#C(\Fq)
\]
so that $\#C(\Fq)\le q$. This contradicts the hypothesis of the lemma.
Therefore, there is at least one value of $a$
for which $D_a$ has genus $2g$ and $\#D_a(\Fq)\ge\#C(\Fq)$. We have already
seen that every such $D_a$ has three rational Weierstrass points, so we are 
done.
\end{proof}

Lemmas~\ref{L:2g1char2} and~\ref{L:2gchar2} give us the machinery with which
to build an induction as in the previous subsection, and once again we are
left with the task of producing curves of genus~$2$ and~$3$ with many points
to use as base cases. Such curves are provided by the following lemma.

\begin{lemma}
\label{L:genus2and3char2}
Let $q$ be a power of $2$ with $q>8$. Then for $g=2$ and $g=3$ there is a 
hyperelliptic curve $C/\Fq$ of genus~$g$ with at least two rational Weierstrass
points and with $\#C(\Fq) > 1 + q + 4\sqrt{q} - 12$.
\end{lemma}

\begin{proof}
We observe that an ordinary elliptic curve over $\Fq$ can be written in the
form $y^2 + y = x + a/x$ if and only if it has a rational point of
order $4$, one such point being $(\sqrt{a},0)$. Furthermore, such a curve
has a rational point of order $8$ if and only if $a$ has absolute trace~$0$; 
in this case, if we write $a = b^2 + b$, then an $8$-torsion point is
given by $((a^4 + a^3)^{1/4}, b^{1/4})$.

Let $t_0$ and $t_4$ be the largest integers less than or equal to $2\sqrt{q}$ 
such that $1 + q + t_0\equiv 0\bmod 8$ and $1 + q + t_4\equiv 4\bmod 8$, so that 
$t_i>2\sqrt{q} - 8$ for both values of $i$ and $t_i>2\sqrt{q} - 4$ for one value
of~$i$. For $i=0$ and $i=4$, let $E_i$ be an ordinary elliptic curve over $\Fq$
with trace~$t_i$. Then we can write the $E_i$ in the form $y^2 + y = x + a_i/x$,
where $a_0$ has absolute trace~$0$ and $a_4$ has absolute trace~$1$.

Let $C$ be the hyperelliptic curve
\begin{equation}
\label{EQ:genus2char2}
y^2 + y = \frac{\sqrt{a_0a_4}}{a_0+a_4} x + \frac{a_0+a_4}{x} + \frac{a_0+a_4}{x+1}\,.
\end{equation}
Clearly $C$ has three rational Weierstrass points. Also, if we set 
$z = y + \sqrt{a_i/(a_0+a_4)} x$, then
\[
z^2 + z = \frac{a_i}{a_0+a_4} (x^2 + x) + \frac{a_0+a_4}{x^2 + x}\,,
\]
and the quotient of this curve by the involution that sends $(x,z)$ to $(x+1,z)$
is clearly 
\[
z^2 + z = \frac{a_i}{a_0+a_4} w + \frac{a_0+a_4}{w}\,,
\]
which we check is isomorphic to $E_i$. Therefore $\Jac C$ is isogenous to 
$E_0\times E_4$, so $\#C(\Fq) = 1 + q + t_0 + t_4 > 1 + q + 4\sqrt{q} - 12.$

Now let $D$ be the double cover of $C$ obtained by adjoining a root of
\begin{equation}
\label{EQ:genus2cover}
u^2 + u = \frac{a_0+a_4}{x}
\end{equation}
to the function field of~$C$. Combining equations~\eqref{EQ:genus2char2} 
and~\eqref{EQ:genus2cover} and writing $v = y + u$, we find that $D$ can be 
written in the form
\begin{equation}
\label{EQ:genus3char2}
v^2 + v = \frac{\sqrt{a_0 a_4}}{u^2 + u} 
          + \frac{a_0^2 + a_4^2}{u^2 + u + a_0 + a_4} + a_0 + a_4\,.
\end{equation}
Because the absolute trace of $a_0 + a_4$ is~$1$, the quadratic twist $D'\to C$
of the double cover $D\to C$ can be given by
\begin{equation}
\label{EQ:genus3char2bis}
v^2 + v = \frac{\sqrt{a_0 a_4}}{u^2 + u + a_0 + a_4} 
         + \frac{a_0^2 + a_4^2}{u^2 + u} + a_0 + a_4\,.
\end{equation}
One of the two curves $D$ and $D'$ will have at least as many points as~$C$, and
$D$ and $D'$ each have exactly two rational Weierstrass points: On both curves, 
the points with $u=0$ and $u=1$ are Weierstrass points. Thus, one of $D$ and
$D'$ will be a hyperelliptic curve of genus~$3$ with the desired properties.
\end{proof}

\begin{proof}[Proof of Theorem~\textup{\ref{T:explicit}(ii)}]
For $q=2$, $4$, and $8$, we need to show that there are hyperelliptic curves
of every genus $g>1$ with at least $4$, $9$, and $16$ points, respectively. 
For $q=2$ and $q = 4$ this follows from Lemma~\ref{L:polynomials} below.
For $q=8$, Lemma~\ref{L:polynomials} gives us the desired hyperelliptic
curve when $g\ge 4$, so we are left to find examples of genus~$2$ and genus~$3$.
These are provided by the genus-$2$ curve $y^2 + y = x^5 + x^3$ and the 
genus-$3$ curve $y^2 + y = rx^7$, where $r$ satisfies $r^3 + r + 1 = 0$; both
of these curves have $17$ rational points.

Now assume that $q>8$.
We prove the result by induction on $g$. The statement is true for $g=2$ and $g=3$
by Lemma~\ref{L:genus2and3char2}. Now suppose Theorem~\ref{T:explicit}(ii)
holds for all $g$ less than some integer $G>3$. We will show it also holds 
when $g = G$. 

Set $h = \lfloor G/2\rfloor$. Then $h>1$, and we may apply
Theorem~\ref{T:explicit}(ii) to show that for every~$q$, there is a 
hyperelliptic curve $C/\Fq$ of genus~$h$ with at least two rational Weierstrass
points and with $\#C(\Fq) > 1 + q + 4\sqrt{q} - 12$.

If $G$ is odd, we apply Lemma~\ref{L:2g1char2} to the curve $C$ and
find that there is a hyperelliptic curve $D$ of genus $2h+1 = G$
with at least two rational
Weierstrass points and with $\#D(\Fq)\ge \#C(\Fq)$.

If $G$ is even, we apply Lemma~\ref{L:2gchar2} to the curve $C$ and find that
there is a hyperelliptic curve $D$ of genus $2h = G$ with at least two rational
Weierstrass points and with $\#D(\Fq)\ge \#C(\Fq)$ or with $\#D(\Fq) = 2q-1$.
\end{proof}

\begin{lemma}
\label{L:polynomials}
Let $q>1$ be a power of $2$ and let $g$ be an integer with $g\ge q/2$. Then
there is a hyperelliptic curve of genus $g$ over $\Fq$ having $2q+1$ rational points.
\end{lemma}

\begin{proof}
Consider the function $\Fq\to\Fq$ that sends $a$ to $a^{2g+1}$. By Lagrange
interpolation, there is a polynomial $f\in\Fq[x]$ of degree at most $q-1$ that agrees
with this function. Therefore, the polynomial $x^{2g+1} + f$ has degree $2g+1$
and evaluates to $0$ for every $x\in\Fq$. Therefore the curve
$y^2 + y = x^{2g+1} + f$ has $2q+1$ rational points and has genus~$g$.
\end{proof}


\subsection{Remarks on constructing examples}
\label{SS:construction}

We mentioned earlier that our arguments in this section suggest an algorithm for
constructing the curves whose existence is asserted by Theorem~\ref{T:explicit}.
In particular, Lemmas~\ref{L:genus2}, \ref{L:genus3}, and~\ref{L:genus2and3char2} provide
explicit ways of finding the base curves of genus $2$ and $3$ with which to
start the construction; these lemmas require that we find an elliptic curve
over $\Fq$ with a certain specific trace and with all of its $2$-torsion points
rational. In general it is a difficult problem to find an elliptic curve over a 
given finite field with a given number of points, and the best general algorithm 
for doing so is essentially to pick elliptic curves at random until one finds
one with the desired order. In our case, the discriminant of the endomorphism
ring of the curves we want is smaller than average --- it's $O(\sqrt{q})$ instead of 
$O(q)$ --- so we might choose to use the Hilbert class polynomial 
method~\cite{Sutherland2011} instead.

Once we have a starting curve, we need to recursively construct curves in the
tower leading to the desired genus. Lemmas~\ref{L:2g1} and~\ref{L:2g1char2} 
tell us how to quickly produce a curve of genus $2g+1$ from a curve of 
genus~$g$, with the new curve having at least as many points as the old;
the work required is simply that of counting points on a single
hyperelliptic curve of genus $2g+1$. It is more difficult to produce a curve
of genus $2g$ from a curve of genus~$g$, with the new curve having at least as 
many points as the old. Lemmas~\ref{L:2g} and \ref{L:2gchar2} tell us that there
is at least one curve in an explicit one-parameter family with the desired 
properties, but in the worst case we might have to count points on essentially
all curves in the family before finding one that works. Heuristically, though, 
we expect to find a good curve after only a few tries.

In the worst case, then, we might need to count points on as many as 
$O(q\log g)$ hyperelliptic curves over $\Fq$ in order to find a curve of
genus~$g$ over $\Fq$ with close to $q + 1 + 4\sqrt{q}$ points. In practice,
we find that many fewer steps are required, and we expect the algorithm sketched
above to require counting points on $O(\log g)$ hyperelliptic curves, once we
have found the base curve of genus~$2$ or~$3$.

We might compare this heuristic complexity to that of the na\"{\i}ve method of
picking hyperelliptic curves of genus~$g$ over $\Fq$ at random until we find
one with close to $q + 1 \pm 4\sqrt{q}$ points. Essentially, the na\"{\i}ve 
method requires waiting for a four-sigma event to occur, so we might expect
to try about $15{,}000$ curves on average before finding one with more than
$q + 1 + 4\sqrt{q}$ points.

\bibliographystyle{alphaurl}

\bibliography{synthbib}

\end{document}